\documentclass{amsart}

\usepackage{hyperref}
\usepackage{tikz-cd}
\usepackage{amsfonts,amsmath,amssymb,mathrsfs, stmaryrd, amsthm}
\usepackage{cite}

\title{Convergence of spectral truncations for compact metric groups}
\author{Yvann Gaudillot--Estrada}

\author{Walter D. van Suijlekom}

\address{Institute for Mathematics, Astrophysics and Particle Physics, Radboud
University Nijmegen, Heyendaalseweg 135, 6525 AJ Nijmegen, The Netherlands.
}
\email{yvann.gaudillot-estrada@ens.psl.eu}
\email{waltervs@math.ru.nl}
\date{October 23, 2023}

\newtheorem{theorem}{Theorem}
\newtheorem{proposition}[theorem]{Proposition}

\theoremstyle{definition}
\newtheorem{definition}[theorem]{Definition}
\newtheorem{remark}[theorem]{Remark}

\newtheorem{example}[theorem]{Example}

\def\hat{\widehat}
\def\A{\mathcal A}
\def\H{L^2(G)}
\def\R{\mathbb{R}}

\begin{document}
\maketitle

\begin{abstract}
  We consider Gromov--Hausdorff convergence of state spaces for spectral truncations of a compact metric group $G$. We work in the context of order-unit spaces and consider orthogonal projections $P_\Lambda$ in $L^2(G)$ corresponding to finite subsets of irreducible representations $\Lambda \subseteq \widehat G$. We then prove that the sequence of truncated state spaces  $\{ S(P_\Lambda C(G) P_\Lambda)\}_\Lambda$ Gromov--Hausdorff converges to the original state space $S(C(G))$, when these are equipped with a metric associated to a Lip-norm which in turn is induced by the action of $G$. 
  \end{abstract}

\section{Introduction}
This paper deals with the question of how one can approximate the geometric structure of a compact metric group $G$ by finite-dimensional, so-called spectral truncations. We work in the context of quantum metric spaces as introduced by Rieffel \cite{Rie00} (and further developed in \cite{Ker03,Lat16,Lat18}) but employ the classical notion of Gromov--Hausdorff convergence for compact metric spaces. The spectral truncations we consider are defined by orthogonal projections onto finite subsets of irreducible representations that appear in the Peter--Weyl decomposition of $L^2(G)$; they yield so-called order-unit spaces. The main advantage of the latter notion is that it allows to speak of state spaces. Moreover, when the order-unit spaces are equipped with a suitable Lip-norm, there is a metric on these state spaces that metrizes the weak$^\ast$ topology. It is in this context that the question concerning Gromov--Hausdorff convergence can be addressed, and, in the present paper, resolved in the case of compact metric groups equipped with a bi-invariant metric.

This work is closely related to some of our previous work on Gromov--Hausdorff convergence for spectral truncations of Dirac operators on the circle \cite{Sui21,Hek21} and tori \cite{Ber19,LS23}. However, here we treat the case of general compact metric groups equipped with a Lip-norm that is naturally induced by the group action of $G$ (as in \cite{Rie98,Rie02}). We leave the more involved spectral truncations associated to a Laplace or Dirac-type operator on $G$ ---and the corresponding induced Lip-norms--- for future research. Also, the main difference with these previous attempts is that we no longer use specifically the formal adjoint of $C(G) \mapsto P_\Lambda C(G) P_\Lambda$ for the reverse map. This freedom enables us to prove a more general convergence result. In fact, we prove that the truncated state space $S(P_\Lambda C(G) P_\Lambda)$ Gromov--Hausdorff converges to the original state space $S(C(G) )$, where the index $\Lambda$ is a finite subset of the set $\hat{G}$ of irreducible representations of $G$. The net of truncations is parametrized by the directed set of finite subsets of $\hat{G}$.

Our result is also closely related to the recent work \cite{Toy23}, which considers discrete groups with polynomial growth. The author uses  quantum Gromov--Hausdorff distance instead of Gromov--Hausdorff distance as we do here, but in fact our present results extend {\em mutatis mutandis} for the former notion of convergence as well. We should note however that the two notions of distance do not coincide in general \cite{KK23}.

Our results mean a great step forward in the study of spectral approximation of geometric spaces, as initiated in \cite{CS19,CS21}. Subsequent results on the quantum Gromov--Hausdorff convergence for quantum groups and/or homogeneous spaces are now of course waiting to be explored. First examples of truncations of a different kind ---so-called Fourier truncations--- have been analyzed in \cite{AK18,AKK22,Rie22}. We expect that the well-known (co)representation theory of quantum groups allows for a straight-forward generalization of our results to the quantum world.

\subsection*{Acknowledgements}
We thank David Kyed and Malte Leimbach for suggestions and comments on a draft of this paper.

\section{Preliminaries}
\subsection{Order-unit spaces} 
We start by recalling some basic notions on order-unit spaces, referring to \cite[Chapter II]{Alf71} for more details. 
\begin{definition}An {\em order-unit space} is a real ordered vector space $A$ together with an archimedean order unit, {\em i.e.} an element $e \in A$ that satisfies
  \begin{enumerate}
  \item for each $a \in A$ there is an $r \in \R$ such that $a \leq r e$;
    \item if $a \in A$ and if $a \leq r e$ for all $r\in \R_+$ then $a \leq 0$. 
    \end{enumerate}
A linear map $\phi : A \to B$ between two order-unit spaces is said to be an {\em order-unit morphism} if it preserves the unit and the order.
\end{definition}


Let $A$ be an order-unit space and let $\mathbb R$ be equipped with its natural structure of order-unit space. 
There exists a unique order-unit morphism $\mathbb{R} \to A$ and it is injective. We always consider $\mathbb{R}$ as embedded in $A$ through that morphism. There is also a natural order-unit norm $\| \cdot \|$ on $A$ defined by:
$$ \|a \| = \inf \, \{ t > 0 : -t \leq a \leq t \} ,\qquad (\forall a \in A).
$$
Note that for $\phi : A \to B$ an order-unit morphism, we have:
$$\|\phi(a)\| \leq \|a\| , \qquad (\forall a \in A)
.$$

\begin{definition}[State space] A {\em state} on an order-unit space $A$ is an order-unit morphism $\sigma: A \to \mathbb{R}$. We denote by $S(A)$ the set of all states on $A$.

\end{definition}

The state space is compact for the weak$^\ast$ topology by a direct application of the Banach-Alaoglu Theorem.

\subsection{Lip-norms and Gromov--Hausdorff convergence of state spaces}

To introduce a notion of distance for state spaces of order-unit spaces, we follow \cite[Section 2]{Rie00} and use Lipschitz semi-norms:

\begin{definition}[Lip-norm] Let $A$ be an order-unit space. A {\em Lip-norm} on $A$ is a seminorm $L$ on $A$ such that
  \begin{enumerate}
  \item for $a \in A$ we have $L(a)=0$ iff $a \in \R$;
  \item the topology on $S(A)$ from the metric defined by
    $$
 d_L(\sigma, \tau) = \sup \, \{ |\sigma(a) - \tau(a) | : a \in A, L(a) \leq 1 \}
 $$
 is the weak$^\ast$ topology.
  \end{enumerate}
  A convenient way to check the second condition is given by \cite[Theorem 1.9]{Rie98}, stating that it holds iff
  \begin{enumerate}
  \item for $a \in A$ we have $L(a)=0$ iff $a \in \R$;
  \item $d_L$ is bounded;
    \item  the set $\mathcal B_1 := \{ a \in A: L(a) \leq 1, \| a \| \leq 1 \}$ is totally bounded in $A$ for $\| \cdot \|$. 
    \end{enumerate}
\end{definition}

Note also that in \cite[Proposition 4]{C89} it is shown that if $ \{ a \in A: L(a) \leq 1 \} / \R $ is bounded, then $d_L$ is a metric.

Now, we can define a refinement of the category of order-unit spaces taking into account this metric aspect:

\begin{definition} A {\em quantum metric space} is a pair $(A,L)$ with $A$ an order-unit space and $L$ a Lip-norm on $A$.

A {\em morphism} $(A, L_A) \to (B, L_B)$ between two quantum metric spaces is by definition an order-unit morphism $\phi : A \to B$ such that:
$$ \quad L_B(\phi(a)) \leq L_A(a); \qquad ( \forall a \in A).$$
\end{definition}
This terminology is justified by the fact that if $(A,L)$ is a quantum metric space, then $(S(A),d_L)$ is a compact metric space (since $d_L$ metrizes the weak$^\ast$ topology). As a matter of fact, we may consider $S$ as a functor from the category of quantum metric spaces, to that of compact metric spaces, and write accordingly $S(A,L) = (S(A),d_L)$ as the state space of $(A,L)$. 



\begin{example}
  Given a  compact metric space $(X,d)$ we may speak about Lipschitz functions as functions $f: X \to \mathbb C$ for which there exists a constant $k >0$ such that:
$$| f(x) - f(y) | \leq k \, d(x,y); \qquad (\forall x,y \in X).$$
  In that case, the lowest constant $k$ verifying the above property is denoted by $L(f)$.
  
The algebra $\mathcal{A}$ of all Lipschitz functions on $X$ is a unital sub-space of $C(X)$ (continuous real-valued functions on $X$) and hence inherits a structure order-unit space. The semi-norm $L$ is actually a Lip-norm on $\mathcal{A}$, so that $(\mathcal{A},L)$ is a quantum metric space. In this case $S(\mathcal A,L) = (\mathcal M(X),d)$, the space of Borel probability measures on $X$ equipped with the Kantorovich--Rubinstein metric ({\em cf.} \cite{Rac91,RR98}). 
\end{example}

We end this preliminary section with the following general result on the Gromov--Hausdorff convergence of state spaces \cite{Rie00} ({\em cf.} \cite[Theorem 5]{Sui21}):

\begin{theorem}\label{GHstate} Let $(A,L_A)$ and $(B,L_B)$ be two quantum metric spaces. Assume that we are given two morphisms of quantum metric spaces $\phi_{AB} : (B, L_B) \to (A,L_A)$ and $\phi_{BA} : (A,L_A) \to (B, L_B)$. Let $\epsilon$ be the lowest constant such that for all $a \in A$ and $b \in B$:
  \begin{align*}
     \| \phi_{AB} \circ \phi_{BA} (a) - a \| &\leq \, \epsilon \, L_A(a)\\
     \| \phi_{BA} \circ \phi_{AB} (b) - b \| &\leq \epsilon \, L_B(b).
     \end{align*}
Then we have the following upper bound for the Gromov--Hausdorff distance between $S(A,L_A)$ and $S(B, L_B)$:
$$d_{GH}( \, S(A,L_A), S(B, L_B) \, ) \leq \epsilon.$$
\end{theorem}
In \cite{Sui21} such a pair of maps $(\phi_{AB}, \phi_{BA})$ was called a {\em $C^1$-approximate order isomorphism}. 

\begin{remark}
  This result actually extends to yield convergence in the quantum Gromov--Hausdorff distance as well, as already exploited in \cite{Toy23}. We stress that the quantum Gromov--Hausdorff distance and the Gromov--Hausdorff in general may not agree, as witnessed by the interesting paper \cite{KK23}. 
  
\end{remark}

\section{Truncations of a compact metric group}

We first recall the notion of metric group:

\begin{definition}
  A metric $d$ on a group $G$ is said to be {\em bi-invariant} if:
  $$
   d(gx,gy)= d(x,y) = d(xg, yg); \qquad (\forall g,x,y\in G).$$
A {\em metric group} is a pair $(G,d)$ with $G$ a group and $d$ a bi-invariant distance on $G$.
\end{definition}

A metric group is in particular a metric space and a topological group. 
Let us assume in the following that $(G,d)$ is a {\em compact metric group}. Then there exists a unique Haar measure on $G$. The Hilbert space of $L^2$ functions with respect to this measure will be denoted simply by $\H$. The left (resp. right) regular action of $G$ on $\H$ is denoted by $U$ (resp. $V$), {\em i.e.} for all $\psi \in \H$ and $g,x \in G$ we set
\begin{align*}
  (U_g\psi)(x) = \psi(g^{-1}x) ,  \qquad
  (V_g\psi)(x) = \psi(xg).
\end{align*}
Recall that the Peter--Weyl Theorem ({\em cf.} \cite[Theorem 1.12]{Kna86}) gives a decomposition of $\H$ into irreducible unitary representations of $G$ with respect to these left and right actions:
\begin{equation}
  \label{peterweyl}
\H = \hat  \bigoplus_{\nu \in \hat G} E_\nu \otimes E_\nu^*,
\end{equation}
where $\hat{G}$ is the set of equivalent classes of irreducible unitary representations of $G$. The equality in \eqref{peterweyl} is intended to stress that in the following we will in fact identify the left- and right-hand side. 


We again write $\mathcal A$ for the algebra of real-valued Lipschitz functions on $(G,d)$.
An element $f$ of $\mathcal{A}$ acts as a bounded self-adjoint operator on $\H$ acting by pointwise multiplication. In other words, $\mathcal{A}$ can also be considered as a unital sub-space of the order-unit space $\mathcal{B}(\H)_{sa}$  of all bounded self-adjoint operators on $\H$. Of course, the order-unit structures on $\mathcal A$ induced by the inclusion in $C(G)$ and $\mathcal{B}(\H)_{sa}$ are the same.

For all $g \in G$ we define:
$$\begin{array}{cccccll} \lambda_g & : & \mathcal{B}(\H)_{sa} & \longrightarrow & \mathcal{B}(\H)_{sa} \\
&& T & \longmapsto & U_g T U_g^* \end{array}$$
and
$$\begin{array}{cccccll} \rho_g & : & \mathcal{B}(\H)_{sa} & \longrightarrow & \mathcal{B}(\H)_{sa} \\
&& T & \longmapsto & V_g T V_g^*. \end{array}$$
The maps $\lambda$ and $\rho$ define two actions  of $G$ on $\mathcal{B}(\H)_{sa}$ by order-unit automorphisms. Furthermore, they preserve $\mathcal{A}$; indeed, they induce the maps given by left/right translation of $G$ on itself, {\em i.e.}
\begin{equation}
  \label{eq:G-action-A}
 \lambda_g(f)(x) = f(g^{-1}x) ,\qquad \rho_g(f)(x) = f(xg) ; \qquad (x,g \in G, f \in \mathcal A).
\end{equation}

For $T \in \mathcal{B}(\H)_{sa}$ we may now define three quantities (that may be infinite):
\begin{align*} &\| T \|_\lambda = \sup_{x \neq y} \, \left\{ \frac{\| \lambda_x(T) - \lambda_y(T) \|}{d(x,y)} \right\} \\
&\| T \|_\rho = \sup_{x \neq y} \, \left\{ \frac{\| \rho_x(T) - \rho_y(T) \|}{d(x,y)} \right\} \\
&\|T\|_{\lambda, \rho} = \max\,  ( \, \| T \|_\lambda , \| T \|_\rho).
\end{align*}
Combining these definitions with Equation \eqref{eq:G-action-A} we conclude that the Lip-norm $L$ on $\mathcal{A}$ can also be expressed as :
\begin{equation}
  \label{otherlipnorm}
 L(f) = \| f \|_{\lambda} = \| f \| _{\rho} ; \qquad (\forall f \in \mathcal A).
\end{equation}

\subsection{Spectral truncations of compact metric groups}
Let us now define the main object of interest to us: the spectral truncations of the metric space $(G,d)$. We would like to start by stressing that we consider truncations induced by projections onto sets of irreducible repesentations of $G$, rather than spectral projections of some Dirac operator, as in previous works \cite{Ber19,CS20,Sui21,Hek21,LS23} (however see Remark \ref{rem:casimir} below).

More precisely, let $\mathcal{F}$ be the directed set of all finite subsets of $\hat{G}$. For any $\Lambda \in \mathcal F$ we may consider the irreducible representations $\nu$ in $\Lambda$ that appear in $L^2(G)$. In view of Equation \eqref{peterweyl} we thus consider the Hilbert subspaces
  $$
\H_\Lambda := \bigoplus_{\nu \in\Lambda} E_\nu \otimes E_\nu^* \subseteq L^2(G).
$$
The orthogonal projections on $\H$ with image $\H_\Lambda$ are denoted by $P_\Lambda$. Clearly, the algebra $\A$ of Lipschitz functions on $G$ does not restrict to act on $\H_\Lambda$; instead, we should {\em compress} $\A$ by $P_\Lambda$ to act on the Hilbert subspace $\H_\Lambda$. More generally, we define an order-unit morphism $\tau_\Lambda : \mathcal{B}(\H)_{sa} \to \mathcal{B}(\H_\Lambda)_{sa}$ by
$$\tau_\Lambda(T) = P_\Lambda T P_\Lambda, \qquad (\forall T \in \mathcal{B}(\H)_{sa}) .$$
Here, we identify $\mathcal{B}(\H_\Lambda)$ with the space of bounded self-adjoint operators on $\H$ commuting with $P_\Lambda$. Since $P_\Lambda$ commutes with $U_g$ and $V_g$ for all $g \in G$, the order-unit morphism $\tau_\Lambda$ commutes with the actions $\lambda$ and $\rho$. We denote by $\mathcal{A}_\Lambda = P_\Lambda \A P_\Lambda$ the image of $\mathcal{A}$ by $\tau_\Lambda$. Then both $\mathcal{B}(\H)_{sa}$ and $\mathcal{A}_\Lambda$ are preserved by $\lambda$ and $\rho$.

\begin{proposition} \label{lipnorm} 
  Let $L_\Lambda$ be the restriction of $\| \cdot \|_{\lambda , \rho}$ to $\mathcal{A}_\Lambda$. Then $L_\Lambda$ is a Lip-norm on $\mathcal{A}_\Lambda$ and $\tau_\Lambda : (\mathcal{A},L) \to (\mathcal{A}_\Lambda, L_\Lambda)$ is a morphism of quantum metric spaces.
\end{proposition}

\begin{proof}
  Since $\mathcal{A}_\Lambda$ is finite dimensional total boundedness of the Lip-ball $\mathcal B_1$ in $\A_\Lambda$ is automatic. It remains to prove that $ L_\Lambda (T) = 0$ implies $T \in \mathbb{C}$ for all $T \in \mathcal A_\Lambda$ (the other implication being obvious). By definition of $L_\Lambda$ in terms of $\| \cdot \|_{\lambda,\rho}$ we have $L_\Lambda(T) =0 $ iff $\lambda_x(T) = T$ for all $x \in G$. Upon writing $T = \tau_\Lambda (f)$ we then have
  $$
T = \int \lambda_x(T) \, d\mu(x) = \tau_\Lambda \left( \int \lambda_x (f) \, d\mu(x) \right)
$$
The $C(G)$-valued integral $\int \lambda_x (f) \, d\mu(x)$ is actually a constant function on $G$ (because of the invariance property of the Haar measure $\mu$) so that $T \in \mathbb R$.

The fact that $\tau_\Lambda : (\mathcal{A},L) \to (\mathcal{A}_\Lambda, L_\Lambda)$ is a morphism of quantum metric spaces is straightforward from Equation \eqref{otherlipnorm}. 
\end{proof}

\begin{remark}
  \label{rem:casimir}
  As mentioned in the introduction, the spectral truncations that we consider in the present paper are not always associated to the spectrum of some Laplace or Dirac-type operator on $G$, in contrast to some of our previous works \cite{Sui21,LS23} (see also \cite{Ber19,Hek21}). Supposing that $G$ is a compact {\em Lie} group (so that Laplace or Dirac-type operators exist) we would expect these to yield a restriction to a subset $\mathcal F'$ of $\mathcal F$ corresponding to the eigenspaces of the pertinent operator. The main challenge is then to realize the above Lip-norms $L_\Lambda$ as the norm of the commutator $\| [D_\Lambda,\cdot ]\|$ with the truncated Dirac operator $D_\Lambda$, as in \cite{CS20}. Nevertheless, the Lip-norm we exploit here is regular in the sense of \cite{Li09,Rie22}, {\em i.e.} it is finite on the linear span of the coordinate elements of all finite-dimensional representations. 
  \end{remark}


\section{Gromov--Hausdorff convergence of state spaces}
We now establish that the net of compact metric spaces $\left( S(\mathcal{A}_\Lambda, L_\Lambda) \right)_{\Lambda \in \mathcal{F}}$ converges to $S(\mathcal{A}, L)$ in Gromov--Hausdorff sense. Our strategy consists in constructing morphisms $\upsilon_\Lambda : (\mathcal{A}_\Lambda, L_\Lambda) \to (\mathcal{A},L) $ and then apply Theorem \ref{GHstate} to the maps $\upsilon_\Lambda$ and $\tau_\Lambda$.

\begin{remark} Note that for any $\Lambda \in \mathcal F$ the order-unit morphism $\tau_\Lambda : \mathcal{A} \to \mathcal{A}_\Lambda$ is surjective. As a consequence, the induced map at the level of state spaces,
$$\begin{array}{cccccll} \tau_\Lambda^* & :& S(\mathcal{A}_\Lambda) & \longrightarrow & S(\mathcal{A}) \\ && \sigma & \longmapsto & \sigma \circ \tau_\Lambda,\end{array}$$
is injective. Since $\mathcal{A}$ is a dense unital sub-space of $C(G)$, the state space $S(\mathcal{A})$ coincides with $S(C(G))$, which which can be identified with the set of Borel probability measures on $G$ . 
\end{remark}

\begin{definition} Let $\Lambda \in \mathcal{F}$. A Borel probability measure $\mu$ on $G$ (considered as an element of $S(\mathcal{A})$ ) is said to be {\bf liftable} (or that it can be lifted) to a state on $\mathcal{A}_\Lambda$ if $\mu \in \tau_\Lambda^*\, S(\mathcal{A}_\Lambda)$. In this case the unique antecedent of $\mu$ by $\tau_\Lambda^*$ will be denoted by $\mu_\Lambda$; it thus satisfies $\mu_\Lambda \circ \tau_\Lambda = \mu$. 
\end{definition}

\begin{remark} If $\mu$ is a Borel probability measure that is liftable to a state on $\mathcal{A}_{\Lambda_0}$ then for all $\Lambda \in \mathcal{F}$ containing $\Lambda_0$, $\mu$ can also be lifted to a state on $\mathcal{A}_\Lambda$ and we have:
$$\mu_\Lambda = \mu_{\Lambda_0} \circ (\tau_{\Lambda_0 |\mathcal{A}_\Lambda}),$$
with $\tau_{\Lambda_0 |\mathcal{A}_\Lambda}$ the restriction of $\tau_\Lambda$ to an order-unit morphism $\mathcal{A}_\Lambda \to\mathcal{A}_{\Lambda_0}$.\end{remark}

\begin{proposition} \label{estimation} Let $\Lambda \in \mathcal{F}$ and $\mu$ a Borel probability measure on $G$ which is liftable to a state on $\mathcal{A}_\Lambda$. Let $\upsilon_\Lambda^\mu : \mathcal{A}_\Lambda \to \mathcal{A}$ be the order-unit morphism defined by:
  $$
  \upsilon_\Lambda^\mu (T)(g) = \mu_\Lambda ( \rho_g(T)) , \qquad (\forall T \in \mathcal{A}_\Lambda, \forall g \in G)
  $$
  Then $\upsilon_\Lambda^\mu : (\mathcal{A}_\Lambda, L_\Lambda) \to (\mathcal{A},L) $ is a morphism of quantum metric spaces.

Moreover, we have
  $$d_{GH}( S(\mathcal{A}_\Lambda, L_\Lambda) , S(\mathcal{A}, L) ) \leq \int_G d(e, x) d\mu(x).$$
\end{proposition}
\begin{proof}
  The fact that $\upsilon_\Lambda^\mu$ is an order-unit morphism is straightforward. For continuity with respect to the Lip-norm we compute
  \begin{align*}
    L( \upsilon_\Lambda^\mu(T)) &= \sup_{x \neq y} \frac{ \left| \upsilon^\mu_\Lambda (T)(x) - \upsilon^\mu_\Lambda (T)(y) \right|}{d(x,y)} \\
      & =  \sup_{x \neq y} \frac {\left| \mu_\Lambda (\rho_x (T) - \rho_y(T) )\right|}{d(x,y)}  \leq \| T \|_ \rho. 
    \end{align*}
This shows that $\upsilon_\Lambda^\mu$ is a morphism of quantum metric spaces.
  
  We establish the second claim by showing that for the pair $(\tau_\Lambda, \upsilon_\Lambda^\mu)$ the constant $\epsilon$ in Theorem \ref{GHstate} has the following upper bound:
$$\epsilon \,  \leq \int_G d(e, x) d\mu(x).$$
We first consider the approximation property of the map $\upsilon_\Lambda^\mu \circ \tau_\Lambda$. Let $f \in \mathcal{A}$ and let us abbreviate $f_\Lambda\equiv \upsilon_\Lambda^\mu \circ \tau_\Lambda (f)$. We have:
$$ f_\Lambda(g) = \int_G f(xg) d\mu(x), \qquad (\forall g \in G),$$
or, equivalently, as a $C(G)$-valued integral:
$$f_\Lambda = \int_G \lambda_{x^{-1}}(f)\, d\mu(x).$$
Hence we have:
\begin{align*} \| f_\Lambda - f \| = & \left\| \int_G  (\lambda_{x^{-1}}(f) - f )d\mu(x) \right\|  \leq \int_G \| \lambda_{x^{-1}}(f) - f \| \, d\mu(x) \\
\leq & \int_G L(f)\,  d(x^{-1}, e)\, d\mu(x) \, = \, L(f) \int_G d(e,x)\, d\mu(x).
\end{align*}
For the composition $\tau_\Lambda \circ \upsilon_\Lambda^\mu$ consider an arbitrary $T = \tau_\Lambda (f) \in \A_\Lambda$. We then have $\upsilon_\Lambda^\mu(T) = f_\Lambda$ in the notation of the previous paragraph, so that 
$$\tau_\Lambda \circ \upsilon_\Lambda^\mu(T) = \int_G \lambda_{x^{-1}}(T) \, d\mu(x),$$
as an operator-valued integral, where we also used that $P_\Lambda$ commutes with $\lambda_{x^{-1}}$ for all $x \in G$.
But then it follows as before that
\begin{align*} \| \tau_\Lambda \circ \upsilon_\Lambda^\mu(T) - T \| = & \left\| \int_G  (\lambda_{x^{-1}}(T) - T )d\mu(x) \right\|  \leq \int_G \| \lambda_{x^{-1}}(T) - T \| \, d\mu(x) \\
\leq & \int_G L(T)\,  d(x^{-1}, e)\, d\mu(x) \, = \, L(T) \int_G d(e,x)\, d\mu(x).
\end{align*}
We thus obtain an upper bound for $\epsilon$ in Theorem \ref{GHstate} from which the claimed upper bound on the Gromov--Hausdorff distance between $S(\A_\Lambda, L_\Lambda)$ and $S(\mathcal A, L)$ directly follows. 
\end{proof}

The remaining task in proving the Gromov--Hausdorff convergence is thus to find a liftable $\mu$ for which the integral  $ \int_G d(e, x) d\mu(x)$ tends to 0. We will take $\mu$ as close as possible to the Dirac mass $\delta_e$ with respect to the weak topology; more precisely, we have:

\begin{proposition} \label{density} Let $S_\mathcal{F} = \bigcup_{\Lambda \in \mathcal{F}}\,\,   \tau_\Lambda^* S(\mathcal{A}_\Lambda) \subseteq S(\mathcal{A})$ be the set of all Borel probability measures that can be lifted to a state on some $\mathcal{A}_\Lambda$. Then $S_\mathcal{F}$ is a dense subset of $S(\mathcal{A})$ for the weak$^\ast$ topology. In particular, for all $\epsilon >0$ there exists $\mu \in S_\mathcal{F}$ such that:
  \begin{equation}
  \int_G d(e,x) d\mu(x) \leq \epsilon.
  \label{eq:dirac-eps}
  \end{equation}
\end{proposition}
\begin{proof}
It is a standard result that density of the convex set $S_{\mathcal F}$ in $S(\mathcal A)$ will follow from the following equality: (see for instance \cite[Theorem 4.3.9]{KR83}):
$$\|f \| = \sup \, \{ | \sigma(f) | : \sigma \in S_\mathcal{F} \} , \qquad ( \forall f \in \mathcal{A}).$$
Let $\H_\mathcal{F} = \sum_{\Lambda \in \mathcal{F}} \H_\Lambda$. Since $\H_\mathcal{F}$ is a dense subspace of $\H$, we have for all $f \in \mathcal{A}$:
$$\| f \| = \sup \, \{ |\langle \psi , f \psi \rangle | : \psi \in \H_\mathcal{F}, \|\psi \| = 1 \}.$$
But $\langle \psi , \cdot \, \psi \rangle \in S_\mathcal{F}$ for all $\psi \in \mathcal{F}$ such that $\| \psi \| = 1$, so that we may conclude that $S_\mathcal{F}$ is a dense subspace of $S(\mathcal{A})$ for the weak* topology.

The state $\mu \in S_{\mathcal F}$ such that Equation \eqref{eq:dirac-eps} holds can then be found as follows. Let $\delta_e$ be the state on $\mathcal{A}$ defined for all $f \in \A$ by $\delta_e(f) = f(e)$.
The function $\Delta : x \mapsto d(e,x)$ is a Lipschitz function, hence an element of $\mathcal{A}$, so by weak* density of $S_\mathcal{F}$ for all $\epsilon>0$ there exists $\mu \in S_{\mathcal {F}}$ such that 
$$ 
\int_G d(e,x) \, d\mu(x) = \left| \mu(\Delta) - \delta_e(\Delta) \right| \leq \epsilon,
$$
as desired.
\end{proof}

Now we can finally state and prove our main Theorem:

\begin{theorem} \label{thm} The net of compact metric spaces $\left( S(\mathcal{A}_\Lambda, L_\Lambda) \right)_{\Lambda \in \mathcal{F}}$ converges to $S(\mathcal{A}, L)$ in Gromov--Hausdorff sense.
\end{theorem}

\begin{proof}
  Let $\epsilon > 0$ and let $\mu \in S_\mathcal{F}$ be a liftable probability measure such as in Proposition \ref{density}, so that Equation \eqref{eq:dirac-eps} holds.  In particular, there is a $\Lambda_0 \in \mathcal{F}$ such that $\mu$ is liftable to a state on $\mathcal{A}_{\Lambda_0}$. Then for all $\Lambda \in \mathcal{F}$ containing $\Lambda_0$, the measure $\mu$ is liftable to a state on $\mathcal{A}_\Lambda$. Applying Proposition \ref{estimation} then yields for all $\Lambda \in \mathcal{F}$ such that $\Lambda_0 \subseteq \Lambda$ we have
  $$
  d_{GH}( S(\mathcal{A}_\Lambda, L_\Lambda) , S(\mathcal{A}, L) ) \leq \epsilon,$$
  which completes the proof. 
\end{proof}

\newcommand{\noopsort}[1]{}\def\cprime{$'$}


\begin{thebibliography}{99}

\bibitem{AK18}
K.~Aguilar and J.~Kaad.
\newblock The {P}odle\'{s} sphere as a spectral metric space.
\newblock {\em J. Geom. Phys.} 133 (2018)  260--278.

\bibitem{AKK22}
K.~Aguilar, J.~Kaad, and D.~Kyed.
\newblock The {P}odle\'{s} spheres converge to the sphere.
\newblock {\em Comm. Math. Phys.} 392 (2022)  1029--1061.

\bibitem{Alf71}
E.~M. Alfsen.
\newblock {\em Compact convex sets and boundary integrals}, volume Band 57 of
  {\em Ergebnisse der Mathematik und ihrer Grenzgebiete [Results in Mathematics
  and Related Areas]}.
\newblock Springer-Verlag, New York-Heidelberg, 1971.

\bibitem{Ber19}
T.~Berendschot.
\newblock Truncated geometry.
\newblock Master's thesis, Radboud University Nijmegen, 2019.

\bibitem{CS19}
A.~Chamseddine and W.~D. van Suijlekom.
\newblock A survey of spectral models of gravity coupled to matter.
\newblock In {\em Advances in noncommutative geometry---on the occasion of
  {A}lain {C}onnes' 70th birthday}, pages 1--51. Springer, Cham, 2019.

\bibitem{C89}
A.~Connes.
\newblock Compact metric spaces, {F}redholm modules, and hyperfiniteness.
\newblock {\em Ergodic Theory Dynam. Systems} 9 (1989)  207--220.

\bibitem{CS20}
A.~Connes and W.~D. van Suijlekom.
\newblock Spectral truncations in noncommutative geometry and operator systems.
\newblock {\em Comm. Math. Phys.} 383 (2021)  2021--2067.

\bibitem{CS21}
A.~Connes and W.~D. van Suijlekom.
\newblock Tolerance relations and operator systems.
\newblock {\em Acta Sci. Math. (Szeged)} 88 (2022)  101--129.

\bibitem{Hek21}
E.~Hekkelman.
\newblock Truncated geometry on the circle.
\newblock {\em Lett. Math. Phys.} 112 (2022)  Paper No. 20, 19.

\bibitem{KK23}
J.~Kaad and D.~Kyed.
\newblock A comparison of two quantum distances.
\newblock {\em Mathematical Physics, Analysis and Geometry} 26 (2023).

\bibitem{KR83}
R.~V. Kadison and J.~R. Ringrose.
\newblock {\em Fundamentals of the theory of operator algebras. {V}ol. {I}},
  volume 100 of {\em Pure and Applied Mathematics}.
\newblock Academic Press Inc. [Harcourt Brace Jovanovich Publishers], New York,
  1983.
\newblock Elementary theory.

\bibitem{Ker03}
D.~Kerr.
\newblock Matricial quantum {G}romov-{H}ausdorff distance.
\newblock {\em J. Funct. Anal.} 205 (2003)  132--167.

\bibitem{Kna86}
A.~W. Knapp.
\newblock {\em Representation theory of semisimple groups}, volume~36 of {\em
  Princeton Mathematical Series}.
\newblock Princeton University Press, Princeton, NJ, 1986.
\newblock An overview based on examples.

\bibitem{Lat16}
F.~Latr\'emoli\`ere.
\newblock The quantum {G}romov-{H}ausdorff propinquity.
\newblock {\em Trans. Amer. Math. Soc.} 368 (2016)  365--411.

\bibitem{Lat18}
F.~Latr\'{e}moli\`ere.
\newblock The {G}romov-{H}ausdorff propinquity for metric spectral triples.
\newblock {\em Adv. Math.} 404 (2022)  Paper No. 108393, 56.

\bibitem{LS23}
M.~Leimbach and W.~D. van Suijlekom.
\newblock Gromov--{H}ausdorff convergence of spectral truncations for
  low-dimensional tori, arXiv:2302.07877.

\bibitem{Li09}
H.~Li.
\newblock Compact quantum metric spaces and ergodic actions of compact quantum
  groups.
\newblock {\em J. Funct. Anal.} 256 (2009)  3368--3408.

\bibitem{Rac91}
S.~T. Rachev.
\newblock {\em Probability metrics and the stability of stochastic models}.
\newblock Wiley Series in Probability and Mathematical Statistics: Applied
  Probability and Statistics. John Wiley \& Sons, Ltd., Chichester, 1991.

\bibitem{RR98}
S.~T. Rachev and L.~R\"{u}schendorf.
\newblock {\em Mass transportation problems. {V}ol. {I}}.
\newblock Probability and its Applications (New York). Springer-Verlag, New
  York, 1998.
\newblock Theory.

\bibitem{Rie22}
M.~A. Rieffel.
\newblock Convergence of {F}ourier truncations for compact quantum groups and
  finitely generated groups.
\newblock arXiv:2210.00387.

\bibitem{Rie98}
M.~A. Rieffel.
\newblock Metrics on states from actions of compact groups.
\newblock {\em Doc. Math.} 3 (1998)  215--229.

\bibitem{Rie00}
M.~A. Rieffel.
\newblock Gromov-{H}ausdorff distance for quantum metric spaces.
\newblock {\em Mem. Amer. Math. Soc.} 168 (2004)  1--65.

\bibitem{Rie02}
M.~A. Rieffel.
\newblock Matrix algebras converge to the sphere for quantum
  {G}romov-{H}ausdorff distance.
\newblock {\em Mem. Amer. Math. Soc.} 168 (2004)  67--91.

\bibitem{Toy23}
R.~Toyota.
\newblock Quantum {G}romov-{H}ausdorff convergence of spectral truncations for
  groups with polynomial growth, 2309.13469.

\bibitem{Sui21}
W.~D. van Suijlekom.
\newblock Gromov-{H}ausdorff convergence of state spaces for spectral
  truncations.
\newblock {\em J. Geom. Phys.} 162 (2021)  Paper No. 104075, 11.

\end{thebibliography}
\end{document}